\documentclass[11pt,reqno]{amsart}
\usepackage[T1]{fontenc}
\usepackage[margin=1.15in]{geometry}
\usepackage{amsmath,amssymb,amsthm,mathtools}
\usepackage{enumitem}
\usepackage[protrusion=true,expansion=false]{microtype}
\usepackage{xcolor}

\usepackage[colorlinks=true,linkcolor=blue,citecolor=blue,urlcolor=blue]{hyperref}
\hypersetup{pdftitle={Chern slopes with prescribed fundamental group},
            pdfauthor={Maycol Falla Luza},
            pdfsubject={14J29, 14J80, 14E20, 14J10}}

\theoremstyle{plain}
\newtheorem{theorem}{Theorem}[section]
\newtheorem{proposition}[theorem]{Proposition}
\newtheorem{lemma}[theorem]{Lemma}
\newtheorem{corollary}[theorem]{Corollary}
\newtheorem{conjecture}[theorem]{Conjecture}
\newtheorem{question}[theorem]{Question}
\theoremstyle{definition}
\newtheorem{definition}[theorem]{Definition}
\newtheorem{construction}[theorem]{Construction}

\numberwithin{equation}{section}

\newcommand{\PP}{\mathbb{P}}
\newcommand{\CC}{\mathbb{C}}
\newcommand{\OO}{\mathcal{O}}
\newcommand{\sig}{\sigma}
\DeclareMathOperator{\Supp}{Supp}
\DeclareMathOperator{\pr}{pr}
\DeclareMathOperator{\Div}{div}

\begin{document}

\title[Chern slopes with prescribed $\pi_1$]
{The Geography of Chern Slopes with Prescribed Fundamental Group}

\author{Maycol Falla Luza}
\address{Instituto de Matem\'atica e Estat\'istica, Universidade Federal Fluminense,
Rua Professor Marcos Waldemar de Freitas Reis s/n, 24210-201 Niter\'oi, RJ, Brazil}
\email{hfalla@id.uff.br}


\begin{abstract}
Let $G$ be the topological fundamental group of a nonsingular complex projective surface.
Troncoso and Urz\'ua proved that the Chern slopes $c_1^2/c_2$ of minimal surfaces of general
type $S$ with $\pi_1(S)\simeq G$ are dense in $[1,3]$, and left $[1/3,1)$ open. We prove that
they are dense in $[1/2,3]$, an interval that cannot be enlarged without contradicting either
a theorem of Mendes Lopes and Pardini or Reid's conjecture. We prove more: the slopes of such
surfaces with $K_S$ \emph{ample} are dense in $[1/2,2]$, the first case in which the
conjecture of Troncoso and Urz\'ua on ample canonical classes is established. The tool is an
exact ampleness criterion for their product construction, which shows in particular that
their own surfaces never have ample canonical class, whatever the defining sections.
\end{abstract}

\maketitle

\section{Introduction}\label{sec:intro}

\subsection{The geography problem with prescribed \texorpdfstring{$\pi_1$}{pi\_1}}
We work over $\CC$. For a minimal nonsingular projective surface $S$ of general type, the
Chern numbers $c_1^2(S)=K_S^2$ and $c_2(S)=e(S)$ are positive and satisfy the Noether
inequality $\tfrac15(c_2-36)\le c_1^2$ and the Bogomolov--Miyaoka--Yau inequality
$c_1^2\le3c_2$. The \emph{Chern slope} of $S$ is the ratio $c_1^2(S)/c_2(S)$.

The geography problem asks which slopes occur under prescribed constraints. Roulleau and
Urz\'ua \cite{RU} proved that the slopes of \emph{simply connected} minimal surfaces of
general type are dense in $[2,3]$; combined with earlier constructions
\cite{Per,Ch,PPX,Urz10} this gives density in $[1/5,3]$, the largest interval allowed by the
two inequalities above. Troncoso and Urz\'ua \cite{TU} fixed the fundamental group instead:

\begin{theorem}[{\cite[Thm.~1.1]{TU}}]\label{thm:TUmain}
Let $G$ be the topological fundamental group of a nonsingular complex projective surface.
Then the Chern slopes of minimal nonsingular projective surfaces of general type $S$ with
$\pi_1(S)\simeq G$ are dense in $[1,3]$.
\end{theorem}

By the Lefschetz hyperplane theorem the groups $G$ so admissible are exactly the fundamental
groups of nonsingular projective varieties, a class containing every finite group by Serre's
construction \cite{Ser}; the theorem is therefore a statement about a very large family.

The lower endpoint $1$ is not known to be optimal, but little room is left below it. Mendes
Lopes and Pardini \cite{MLP} imply that a minimal surface of general type with
$c_1^2<\frac13c_2$ and finite fundamental group has $|\pi_1|\le9$; as every finite group is
the fundamental group of a nonsingular projective surface, density in $[1/5,3]$ for arbitrary
$G$ is false and $[1/3,3]$ is the largest conceivable interval. Second, what may happen
below $1/2$ is governed by the following.

\begin{conjecture}[Reid; see {\cite[p.~294]{BHPV}}]\label{conj:reid}
Let $S$ be a minimal surface of general type with $c_1^2(S)<\frac12c_2(S)$. Then $\pi_1(S)$
is either finite or commensurable with the fundamental group of a compact Riemann surface.
\end{conjecture}

Taking for $G$ the fundamental group of a fake projective plane, which is neither, a
density statement below $1/2$ for arbitrary $G$ would contradict Conjecture
\ref{conj:reid}. Thus $[1/2,3]$ is the largest interval compatible with current
expectations. The value $1/2$ is moreover sharp for the conjecture itself: Keum
\cite{Keum} constructs a minimal surface of general type with $c_1^2=4$ and $c_2=8$ ---
slope exactly $1/2$ --- whose fundamental group $\mathbb{Z}^4\rtimes(\mathbb{Z}/2\mathbb{Z})^2$
is infinite and not commensurable with the fundamental group of a curve.

All the surfaces produced in \cite{RU,TU} contain $(-2)$-curves, so their canonical classes
are big and nef but not ample. Accordingly \cite{TU} pose:

\begin{conjecture}[{\cite[Conj.~4.5]{TU}}]\label{conj:45}
For $G$ as above, the Chern slopes of minimal nonsingular projective surfaces of general type
with $\pi_1(S)\simeq G$ and ample canonical class are dense in $[1,3]$.
\end{conjecture}

\subsection{Results}
Our first result extends the interval of Theorem \ref{thm:TUmain} from $[1,3]$ to $[1/2,3]$.

\begin{theorem}\label{thm:half3}
Let $G$ be the topological fundamental group of a nonsingular complex projective surface.
Then the Chern slopes $c_1^2(S)/c_2(S)$ of minimal nonsingular projective surfaces $S$ of
general type with $\pi_1(S)\simeq G$ are dense in $[1/2,3]$.
\end{theorem}

Of the range $[1/3,1)$ left open in \cite{TU}, this settles $[1/2,1)$ and leaves only
$[1/3,1/2)$, which is exactly the region governed by Reid's conjecture.

Theorem \ref{thm:half3} is deduced from Theorem \ref{thm:TUmain} together with the following
sharper statement, which is where the work lies. It carries the additional conclusion that
the canonical class is ample --- a property that no surface produced in \cite{RU,TU} has.

\begin{theorem}\label{thm:main}
Let $G$ be as above. Then the Chern slopes $c_1^2(S)/c_2(S)$ of minimal nonsingular
projective surfaces $S$ of general type with
\[
\pi_1(S)\simeq G\qquad\text{and}\qquad K_S\ \text{ample}
\]
are dense in $[1/2,2]$.
\end{theorem}

For $G$ trivial the slope range of Theorem \ref{thm:main} was already reached by
Persson \cite[Thm.~3]{Per}, who realises every \emph{pair} $(\chi,c_1^2)$ in a region
covering it by a simply connected minimal surface of general type; what is new here is the
prescribed fundamental group together with the ample canonical class.

Theorem \ref{thm:main} establishes Conjecture \ref{conj:45} on $[1,2]$ --- the first case in
which it is known --- and extends it below its stated range, down to $1/2$. On $[2,3]$ the conjecture remains open, and \S\ref{sec:limits} reduces it there to a single
question about the surfaces of \cite{RU}.

Everything goes through the
construction of \cite[\S4]{TU}, recalled in \S\ref{sec:prelim}: given nonsingular projective
surfaces $X,Y$ with $K_X,K_Y$ nef and $K_X^2>0$, a very ample line bundle $B$ on $Y$ and a
lef line bundle $\sig$ of exponent $1$ on $X$, one obtains a nonsingular projective surface
$S\subset X\times Y$, cut out by two general members of $|M|$ with
$M=\pr_1^*\sig\otimes\pr_2^*B$, satisfying $\pi_1(S)\simeq\pi_1(X)\times\pi_1(Y)$ and $K_S$
big and nef. What that construction leaves undecided is when $K_S$ is ample.

\begin{theorem}\label{thm:criterion}
In the above situation, $K_S$ is ample if and only if $X$ contains no irreducible curve
$\Gamma$ with
\[
K_X\cdot\Gamma=0\quad\text{and}\quad\sig\cdot\Gamma=0 .
\]
\end{theorem}

The two directions play different roles. Sufficiency is all that Theorem \ref{thm:main} needs,
and only through the special case that $K_X$ ample implies $K_S$ ample (Corollary
\ref{cor:easy}); it is an application of the Nakai--Moishezon criterion. Necessity is the
substantial half, and it gives:

\begin{corollary}\label{cor:never}
For the surfaces $S_p$ constructed in \cite[Thm.~4.3]{TU}, the canonical class $K_{S_p}$ is
never ample, for any choice of the defining sections.
\end{corollary}

The ingredients of Theorem \ref{thm:criterion} are classical, but neither the equivalence
itself nor the unconditional obstruction of Corollary \ref{cor:never} appears in the
literature.

\subsection{Strategy}\label{ss:strategy}
Two changes to \cite{TU} suffice. The first is the input. There, $X_p$ is a $p$-th root cover
branched along a large arrangement of curves, and its $(-2)$-curves come from resolving the
cyclic quotient singularities over the nodes of that arrangement; branching instead along a
\emph{smooth} curve leaves no singularity to resolve, hence no $(-2)$-curve at all, and makes
$K_X$ the pullback of an ample class under a finite morphism, hence ample. The simplest such
input is a double cover of $\PP^1\times\PP^1$.

The second is the auxiliary bundle. In \cite{TU} one has $\sig^2=p$ and
$\sig\cdot K_{X_p}=O(p^3)$ against $c_i(X_p)=O(p^5)$, so every term involving $\sig$ is
asymptotically negligible and the slope of $S$ is that of $X$ in the limit. There is no
reason to impose this: taking $B=mB_0$ with $B_0$ very ample and $m\to\infty$, the formulas
of \cite[Thm.~4.2]{TU} give
\[
\lim_{m\to\infty}\frac{c_1^2(S)}{c_2(S)}=\Lambda(X,\sig):=
\frac{c_1^2(X)+24\sig^2+12\,\sig\cdot K_X}{c_2(X)+18\sig^2+6\,\sig\cdot K_X},
\]
independently of $Y$ and $B_0$, and letting $\sig$ be comparable in size to $K_X$ makes
$\Lambda$ a two-parameter quantity. Our input surfaces all have slope $<1/2$: the entire
interval $[1/2,2]$ is produced by the terms that \cite{TU} discard.

\section{The product construction}\label{sec:prelim}

We recall the relevant notions and the theorem of \cite{TU} on which everything rests.
Throughout, $\pr_1$ and $\pr_2$ denote the projections of a product of two surfaces, and we
do not distinguish notationally between a line bundle and the associated divisor class.

\begin{definition}[{\cite{dCM}, \cite[Def.~2.1, Def.~2.5]{TU}}]\label{def:lef}
A proper surjective morphism $\phi:X\to W$ of irreducible varieties is \emph{semismall} if
$\dim\{w\in W:\dim\phi^{-1}(w)=k\}+2k\le\dim X$ for all $k\ge0$. A line bundle $M$ on a
nonsingular projective variety $X$ is \emph{lef}\footnote{An acronym for \emph{Lefschetz
effettivamente funziona}; see \cite{dCM}.} if $|nM|$ is base point free and the associated
morphism is semismall onto its image for some $n>0$; the least such $n$ is the \emph{exponent}
$\exp(M)$.
\end{definition}

If $L$ is ample then $L$ is lef, and $\exp(L)=1$ if $L$ is very ample. We record one
consequence of the definition, used repeatedly below: \emph{if $\exp(M)=1$ then $M$ is
globally generated}. The following is the source of the lef bundles used here and in \cite{TU}.

\begin{proposition}[{\cite[Lem.~2.2, Prop.~2.6]{TU}}]\label{prop:lefcrit}
Let $g:X\to Z'$ be a generically finite morphism of nonsingular projective surfaces and let
$L'$ be very ample on $Z'$. Then $g^*L'$ is lef with $\exp(g^*L')=1$.
\end{proposition}

Indeed $g$ is proper, and generic finiteness forces $\dim g(X)=2$, so $g(X)=Z'$; a surjective
morphism of surfaces has positive-dimensional fibres over a finite set only, hence is
semismall for dimension reasons \cite[Lem.~2.2]{TU}, and \cite[Prop.~2.6]{TU} applies. The underlying Lefschetz theorem for the homotopy groups of a hyperplane section of
a semismall morphism is that of Goresky--MacPherson \cite[Part~II, Thm.~1.1]{GM}.

\begin{theorem}[{\cite[Thm.~4.2]{TU}}]\label{thm:TU42}
Let $X$ and $Y$ be nonsingular projective surfaces with $K_X$ and $K_Y$ nef and $K_X^2>0$.
Let $B$ be a very ample line bundle on $Y$ and $\sig$ a lef line bundle on $X$ with
$\exp(\sig)=1$, and set $M:=\pr_1^*\sig\otimes\pr_2^*B$ on $X\times Y$. Then $M$ is lef of
exponent $1$, and for general $E,E'\in|M|$ the intersection $S:=E\cap E'$ is a nonsingular
projective surface with the following properties.
\begin{enumerate}[label=\textup{(\arabic*)}]
\item $\pi_1(S)\simeq\pi_1(X)\times\pi_1(Y)$.
\item Setting $c(\sig,B)=\frac72\sig^2B^2+\frac32(\sig\cdot K_X)B^2+\frac32(B\cdot K_Y)\sig^2
+\frac12(\sig\cdot K_X)(B\cdot K_Y)$,
\begin{align*}
c_1^2(S)&=c_1^2(X)B^2+c_1^2(Y)\sig^2+8c(\sig,B)-4\sig^2B^2,\\
c_2(S)&=c_2(X)B^2+c_2(Y)\sig^2+4c(\sig,B)+4\sig^2B^2 .
\end{align*}
\item $K_S$ is big and nef.
\end{enumerate}
\end{theorem}

\begin{lemma}\label{lem:Lambda}
In the situation of Theorem \ref{thm:TU42}, let $B_0$ be very ample on $Y$ and $B=mB_0$. Then
\begin{equation}\label{eq:Lambda}
\lim_{m\to\infty}\frac{c_1^2(S)}{c_2(S)}=\Lambda(X,\sig):=
\frac{c_1^2(X)+24\sig^2+12\,\sig\cdot K_X}{c_2(X)+18\sig^2+6\,\sig\cdot K_X},
\end{equation}
a limit independent of $Y$ and $B_0$.
\end{lemma}

\begin{proof}
Expand $c(\sig,B)$ in Theorem \ref{thm:TU42}(2) and collect the terms carrying $B^2$. Since
$8c(\sig,B)-4\sig^2B^2$ contributes $24\sig^2+12\,\sig\cdot K_X$ to the coefficient of $B^2$,
and $4c(\sig,B)+4\sig^2B^2$ contributes $18\sig^2+6\,\sig\cdot K_X$, one gets
\[
c_1^2(S)=B^2\big[c_1^2(X)+24\sig^2+12\,\sig\cdot K_X\big]+R_1,\qquad
c_2(S)=B^2\big[c_2(X)+18\sig^2+6\,\sig\cdot K_X\big]+R_2,
\]
where $R_1$ and $R_2$ are the remaining terms, namely $\sig^2c_1^2(Y)+(B\cdot K_Y)
[12\sig^2+4\,\sig\cdot K_X]$ and $\sig^2c_2(Y)+(B\cdot K_Y)[6\sig^2+2\,\sig\cdot K_X]$. For
$B=mB_0$ one has $B^2=m^2B_0^2$ while $B\cdot K_Y=m(B_0\cdot K_Y)$, so $R_1$ and $R_2$ are
$O(m)$ and the bracketed terms dominate. It remains to check that the second bracket is
positive: $K_X$ is nef with $K_X^2>0$, so $X$ is minimal of general type and $c_2(X)>0$ by
the Bogomolov--Miyaoka--Yau inequality, while $\sig$ is nef by Definition \ref{def:lef},
whence $\sig^2\ge0$ and $\sig\cdot K_X\ge0$.
\end{proof}

\section{Forced curves and an ampleness criterion}\label{sec:criterion}

Throughout this section $X,Y$ are nonsingular projective surfaces, $\sig$ is a globally
generated --- hence nef --- line bundle on $X$, $B$ is a line bundle on $Y$ with $B^2>0$, and
$M=\pr_1^*\sig\otimes\pr_2^*B$. Ampleness of $\sig$ is not assumed.

\begin{lemma}\label{lem:triv}
Let $\Gamma\subset X$ be an irreducible curve with $\sig\cdot\Gamma=0$. Then
$\sig|_\Gamma\simeq\OO_\Gamma$.
\end{lemma}

\begin{proof}
Being the restriction of a globally generated bundle, $\sig|_\Gamma$ is globally generated,
so its global sections define a morphism $\phi:\Gamma\to\PP^n$ with
$\sig|_\Gamma\simeq\phi^*\OO(1)$. Then $\deg\phi^*\OO(1)=\sig\cdot\Gamma=0$, so $\phi(\Gamma)$
cannot be a curve; it is a point, and $\sig|_\Gamma\simeq\OO_\Gamma$.
\end{proof}

\begin{proposition}[Forced curves]\label{prop:forced}
Let $\Gamma\subset X$ be an irreducible curve with $\sig\cdot\Gamma=0$, and let $E,E'$ be the
divisors of arbitrary nonzero sections $s,s'\in H^0(X\times Y,M)$. Then there exists $y_0\in
Y$ with $\Gamma\times\{y_0\}\subset E\cap E'$.
\end{proposition}

\begin{proof}
Fix $x_0\in\Gamma$ and a trivialisation
$\tau:\sig|_\Gamma\xrightarrow{\ \sim\ }\OO_\Gamma$ afforded by Lemma \ref{lem:triv};
evaluating $\tau$ at $x_0$ trivialises the fibre $\sig_{x_0}$, and we use that induced
identification below. For $y\in Y$ write $s_y:=s|_{X\times\{y\}}\in H^0(X,\sig)$. Through
$\tau$, the restriction $s_y|_\Gamma$ becomes a global section of $\OO_\Gamma$, that is a
constant, since $\Gamma$ is projective, reduced and irreducible; a constant is determined by
its value at any one point, so its vanishing is detected at $x_0$ alone. Hence
\[
\Gamma\times\{y\}\subset E\iff s_y|_\Gamma\equiv0\iff s(x_0,y)=0 ,
\]
the last condition being independent of any choice. Restricting $s$ to
$\{x_0\}\times Y\simeq Y$, where $M$ restricts to $\sig_{x_0}\otimes B$, produces
$t:=s(x_0,\cdot)\in H^0(Y,B)$; another trivialisation of $\sig_{x_0}$ would rescale $t$ and
leave $\Div(t)$ unchanged. If $t\equiv0$ then
$Z(s):=\{y\in Y:\Gamma\times\{y\}\subset E\}$ is all of $Y$; otherwise $Z(s)=\Supp\Div(t)$,
where $\Div(t)$ is an effective divisor in $|B|$, nonzero because $B^2>0$. The same applies
to $s'$. If either $Z(s)$ or $Z(s')$ is all of $Y$ we are done, the other being nonempty.
Otherwise $\Div(t)\cdot\Div(t')=B^2>0$, so the two divisors meet, either in a common
component or in $B^2$ points counted with multiplicity. Any $y_0\in Z(s)\cap Z(s')$ has the
required property.
\end{proof}

\begin{proof}[Proof of Theorem \ref{thm:criterion}]
By Theorem \ref{thm:TU42}(3) the class $K_S$ is big and nef, so $K_S^2>0$, and by the
Nakai--Moishezon criterion it remains to determine the sign of $K_S\cdot C$ for irreducible
curves $C\subset S$. Since $K_{X\times Y}\sim\pr_1^*K_X+\pr_2^*K_Y$, adjunction applied twice
gives $K_S\sim\pr_1|_S^*(K_X+2\sig)+\pr_2|_S^*(K_Y+2B)$, so by the projection formula
\begin{equation}\label{eq:KSC}
K_S\cdot C=d_1\big(\pr_1(C)\cdot(K_X+2\sig)\big)+d_2\big(\pr_2(C)\cdot(K_Y+2B)\big),
\end{equation}
each summand being present when the corresponding image is a curve, with $d_1,d_2\ge1$ the
degrees of the induced maps; both summands are $\ge0$, since $K_X$, $\sig$, $K_Y$ and $B$ are
nef.

Suppose first that $X$ contains no curve $\Gamma$ with $K_X\cdot\Gamma=\sig\cdot\Gamma=0$. If
$\pr_2(C)$ is a curve then the second summand of \eqref{eq:KSC} is at least
$2B\cdot\pr_2(C)>0$, since $B$ is very ample and $K_Y$ is nef; hence $K_S\cdot C>0$. If
$\pr_2(C)$ is a point then $C\subset X\times\{y\}$ for some $y$, the projection $\pr_1$ maps
$C$ isomorphically onto $\Gamma:=\pr_1(C)$, and $K_S\cdot C=\Gamma\cdot(K_X+2\sig)$. As $K_X$
and $\sig$ are nef and, by hypothesis, not both trivial on $\Gamma$, this is positive.
Therefore $K_S\cdot C>0$ for all $C$, and $K_S$ is ample.

Conversely, suppose $\Gamma\subset X$ is irreducible with $K_X\cdot\Gamma=\sig\cdot\Gamma=0$.
Since $\exp(\sig)=1$, the bundle $\sig$ is globally generated, so Proposition
\ref{prop:forced} applies and yields $y_0\in Y$ with $C_0:=\Gamma\times\{y_0\}\subset S$.
Then $K_S\cdot C_0=\Gamma\cdot(K_X+2\sig)=0$ by \eqref{eq:KSC}, so $K_S$ is not ample. Note
that Proposition \ref{prop:forced} needs only $\sig\cdot\Gamma=0$, which is what makes
Corollary \ref{cor:never} independent of the choice of sections.
\end{proof}

\begin{proof}[Proof of Corollary \ref{cor:never}]
In the construction of \cite{RU,TU} the branch divisor has
$t_{2,1}=6\beta^4p^4+36\alpha^2\beta^2p^4+\cdots>0$ nodes at which the two branch
multiplicities coincide \cite[\S3]{TU}. Over such a node lies a cyclic quotient singularity of
type $\frac1p(1,p-1)$, whose minimal resolution is a chain of $p-1$ smooth rational curves of
self-intersection $-2$. Let $R$ be one of them. Then $K_{X_p}\cdot R=0$ by adjunction, and
$R$ is contracted by the morphism $f:X_p\to Y_n$, so with $\sig_p=f^*L$ the projection
formula gives $\sig_p\cdot R=L\cdot f_*R=0$. Theorem \ref{thm:criterion} applies.
\end{proof}

\begin{corollary}\label{cor:easy}
In the situation of Theorem \ref{thm:TU42}, if $K_X$ is ample then $K_S$ is ample.
\end{corollary}

\begin{proof}
If $K_X$ is ample then $K_X\cdot\Gamma>0$ for every irreducible curve $\Gamma\subset X$, so
no curve satisfies the condition of Theorem \ref{thm:criterion}.
\end{proof}

\section{The input surfaces}\label{sec:input}

We use the following classical construction. Let $Z$ be a nonsingular projective surface, $L$
a line bundle on $Z$ and $D\in|2L|$ a smooth curve. Inside the total space of $L$, the zero
locus of $z^2-s_D$, where $z$ is the tautological section and $s_D$ defines $D$, is a
nonsingular projective surface $X$ carrying a finite degree-two morphism $\pi:X\to Z$
branched exactly along $D$, and
\begin{equation}\label{eq:dc}
K_X=\pi^*(K_Z+L),\qquad e(X)=2e(Z)-e(D).
\end{equation}
The first identity is Riemann--Hurwitz, the second the additivity of the Euler
characteristic. We record the two properties that matter here.

\begin{lemma}\label{lem:dcprops}
With the above notation:
\begin{enumerate}[label=\textup{(\alph*)}]
\item if $K_Z+L$ is ample, then $K_X$ is ample;
\item if $Z$ is simply connected and $D$ is smooth and ample, then $X$ is simply connected.
\end{enumerate}
\end{lemma}

\begin{proof}
(a) The morphism $\pi$ is finite, and the pullback of an ample line bundle under a finite
morphism is ample; apply this to $K_Z+L$ using \eqref{eq:dc}. (b) This is Cornalba
\cite{Cor}.
\end{proof}

We take $Z=\PP^1\times\PP^1$. Its N\'eron--Severi group is generated by the classes $F_1,F_2$
of the two rulings; we write $(m,n):=mF_1+nF_2$, so that
\begin{gather*}
(m,n)\cdot(m',n')=mn'+m'n,\qquad K_Z=(-2,-2),\qquad K_Z^2=8,\qquad e(Z)=4,
\end{gather*}
and $(m,n)$ is ample if and only if it is very ample, if and only if $m\ge1$ and $n\ge1$. The
surface $Z$ is simply connected.

\begin{construction}\label{constr}
Let $u,v,a,b$ be integers with $u,v,a,b\ge1$. On $Z=\PP^1\times\PP^1$ set
\[
L:=(u+2,\,v+2),\qquad D\in|2L|\ \text{a general member},
\]
let $\pi:X\to Z$ be the double cover branched along $D$, and put
\[
W:=K_Z+L=(u,v),\qquad \sig:=\pi^*\big((a,b)\big).
\]
\end{construction}

\begin{proposition}\label{prop:props}
In Construction \ref{constr}, $D$ is smooth and irreducible and:
\begin{enumerate}[label=\textup{(\alph*)}]
\item $X$ is a nonsingular projective surface with $K_X=\pi^*W$ ample; in particular $X$ is
minimal of general type, $K_X^2>0$, and $X$ contains no $(-2)$-curve;
\item $\pi_1(X)=1$;
\item $\sig$ is ample, globally generated, and lef with $\exp(\sig)=1$;
\item $c_1^2(X)=4uv$ and $c_2(X)=8uv+12u+12v+24$;
\item $\sig^2=4ab$ and $\sig\cdot K_X=2(av+bu)$.
\end{enumerate}
\end{proposition}

\begin{proof}
Since $u,v\ge1$, the class $2L=(2u+4,2v+4)$ is very ample, so by Bertini a general member
$D\in|2L|$ is smooth and irreducible, and $X$ is nonsingular.

(a) By \eqref{eq:dc}, $K_X=\pi^*(K_Z+L)=\pi^*W$ with $W=(u,v)$ ample; apply Lemma
\ref{lem:dcprops}(a). A surface with ample canonical class is minimal of general type and
contains no $(-2)$-curve.

(b) $Z$ is simply connected and $D$ is smooth and ample; apply Lemma \ref{lem:dcprops}(b).

(c) Since $a,b\ge1$, the class $(a,b)$ is very ample on $Z$. The morphism $\pi$ is finite
between nonsingular projective surfaces, hence proper, surjective and generically finite, so
Proposition \ref{prop:lefcrit} applies with $g=\pi$ and $L'=(a,b)$ and yields that
$\sig=\pi^*(a,b)$ is lef with $\exp(\sig)=1$. Being the pullback of an ample, globally
generated class under a finite morphism, $\sig$ is moreover ample and globally generated. Of
these properties only $\exp(\sig)=1$ is used in the sequel; ampleness of $\sig$ is never
needed.

(d) $c_1^2(X)=K_X^2=(\pi^*W)^2=\deg(\pi)\,W^2=2\cdot 2uv=4uv$. Writing $l_1=u+2$ and
$l_2=v+2$, \eqref{eq:dc} gives $e(X)=2\cdot4-e(D)=8-(2-2g(D))=6+2g(D)$, while adjunction on
$Z$ gives
\[
2g(D)-2=D\cdot(D+K_Z)=(2l_1,2l_2)\cdot(2l_1-2,\,2l_2-2)=8l_1l_2-4l_1-4l_2 ,
\]
so $g(D)=1+4l_1l_2-2l_1-2l_2$ and $c_2(X)=8+8l_1l_2-4l_1-4l_2$. Substituting $l_1=u+2$,
$l_2=v+2$ yields $c_2(X)=8uv+12u+12v+24$.

(e) $\sig^2=\deg(\pi)\,(a,b)^2=2\cdot2ab=4ab$, and
$\sig\cdot K_X=\pi^*(a,b)\cdot\pi^*(u,v)=\deg(\pi)\,(a,b)\cdot(u,v)=2(av+bu)$.
\end{proof}

\begin{proposition}\label{prop:slope}
In Construction \ref{constr},
\[
\Lambda(X,\sig)=\frac{4uv+96ab+24(av+bu)}{8uv+12u+12v+24+72ab+12(av+bu)} .
\]
Fixing $u=A$ and $b=1$ and letting $v\to\infty$ gives
\[
\Lambda(X,\sig)\longrightarrow\Lambda_\infty(A,a):=\frac{4A+24a}{8A+12+12a},
\]
and letting $A,a\to\infty$ with $a/A\to y\in(0,\infty)$ gives
$\Lambda_\infty(A,a)\to\mu(y):=\dfrac{1+6y}{2+3y}$. The function $\mu$ is a strictly
increasing homeomorphism of $(0,\infty)$ onto $(1/2,2)$.
\end{proposition}

\begin{proof}
The displayed formula is \eqref{eq:Lambda} together with Proposition \ref{prop:props}(d),(e).
Substituting $u=A$, $b=1$ and dividing numerator and denominator by $v$ gives
$\Lambda_\infty$. Dividing numerator and denominator of $\Lambda_\infty$ by $A$ and letting
$A\to\infty$ with $a/A\to y$ gives $\mu(y)=\frac{4+24y}{8+12y}$. Finally
\[
\mu'(y)=\frac{6(2+3y)-3(1+6y)}{(2+3y)^2}=\frac{9}{(2+3y)^2}>0,
\]
with $\lim_{y\to0^+}\mu(y)=\frac12$ and $\lim_{y\to\infty}\mu(y)=2$.
\end{proof}

\section{Proof of the main theorem}\label{sec:proof}

\begin{proof}[Proof of Theorem \ref{thm:main}]
Let $G=\pi_1(Y_0)$ with $Y_0$ a nonsingular projective surface. As in \cite[Cor.~4.4]{TU}
there is a nonsingular projective surface $Y$ with $K_Y$ nef and $\pi_1(Y)\simeq G$: since
$\pi_1$ is a birational invariant we may replace $Y_0$ by a minimal model, which we take for
$Y$ if its Kodaira dimension is $\ge0$; otherwise it is $\PP^2$ or geometrically ruled over a
nonsingular projective curve $C$, so that $G\simeq\pi_1(C)$, a group depending only on the
genus of $C$. If that genus is positive we take for $Y$ a surface of \cite[Cor.~6.4]{RU},
which is minimal of general type --- hence has $K_Y$ nef --- and has fundamental group that
of a compact Riemann surface of the given genus; if the genus is $0$ then $G=1$ and any
member of Construction \ref{constr} will do, by Proposition \ref{prop:props}(a),(b). Fix a
very ample line bundle $B_0$ on $Y$.

Let $r\in[1/2,2]$ and $\varepsilon>0$. By Proposition \ref{prop:slope}, $\mu$ is a
homeomorphism of $(0,\infty)$ onto $(1/2,2)$, which is dense in $[1/2,2]$; so there is a
positive rational $y=a_0/A_0$ with $|\mu(y)-r|<\varepsilon/4$. Replacing $(A_0,a_0)$ by
$(NA_0,Na_0)$ for $N\gg0$ gives integers $A,a\ge1$ with
$|\Lambda_\infty(A,a)-\mu(y)|<\varepsilon/4$, and then taking $v\gg0$ gives
$|\Lambda(X,\sig)-\Lambda_\infty(A,a)|<\varepsilon/4$ for the surface $X$ and bundle $\sig$
of Construction \ref{constr} with parameters $(u,v,a,b)=(A,v,a,1)$; all three estimates come
from Proposition \ref{prop:slope}.

By Proposition \ref{prop:props}, $X$ is a nonsingular projective surface with $K_X$ ample ---
in particular nef and with $K_X^2>0$ --- and $\pi_1(X)=1$, while $\sig$ is lef with
$\exp(\sig)=1$. Since $K_Y$ is nef and $mB_0$ is very ample for every $m\ge1$, Theorem
\ref{thm:TU42} produces nonsingular projective surfaces $S_m\subset X\times Y$ with
$\pi_1(S_m)\simeq\pi_1(X)\times\pi_1(Y)\simeq G$, and by Lemma \ref{lem:Lambda} some $m$
satisfies $\big|c_1^2(S_m)/c_2(S_m)-\Lambda(X,\sig)\big|<\varepsilon/4$. The four estimates
together give $\big|c_1^2(S_m)/c_2(S_m)-r\big|<\varepsilon$. Finally $K_{S_m}$ is ample by
Corollary \ref{cor:easy}, so $S_m$ is minimal of general type. As $r$ and $\varepsilon$ were
arbitrary, the slopes are dense in $[1/2,2]$.
\end{proof}

\begin{proof}[Proof of Theorem \ref{thm:half3}]
Write $\mathcal{S}(G)$ for the set of Chern slopes of minimal nonsingular projective surfaces
of general type with fundamental group isomorphic to $G$. By Theorem \ref{thm:main},
$\mathcal{S}(G)$ contains a subset dense in $[1/2,2]$, and by Theorem \ref{thm:TUmain} it
contains a subset dense in $[1,3]$. Hence $\mathcal{S}(G)$ is dense in
$[1/2,2]\cup[1,3]=[1/2,3]$.
\end{proof}

\section{The interval \texorpdfstring{$[2,3]$}{[2,3]}}\label{sec:limits}

By \eqref{eq:Lambda}, $\Lambda(X,\sig)<2$ if and only if $c_1^2(X)<2c_2(X)+12\sig^2$; in
particular $\Lambda<2$ whenever the input surface has slope less than $2$. Our input surfaces
have slope $4uv/(8uv+12u+12v+24)<1/2$ by Proposition \ref{prop:props}(d), so reaching $[2,3]$
requires a different input surface.

Such inputs exist: the surfaces $X_p$ of \cite[Thm.~6.3]{RU}, simply connected and minimal of
general type, with slopes arbitrarily close to $3$. A very ample $\sig$ on $X_p$ satisfies the
criterion of Theorem \ref{thm:criterion}, so $K_S$ is ample and $\pi_1(S)\simeq\pi_1(Y)$; this
recovers Catanese \cite[Lem.~1.1]{Cat}, the route suggested in \cite{TU}.

The obstruction is quantitative. Write $\varepsilon=\sig^2/c_2(X)$ and
$t=\sig\cdot K_X/c_2(X)$; by \eqref{eq:Lambda}, an input of slope $3$ gives
\[
\Lambda(X,\sig)=\frac{3+24\varepsilon+12t}{1+18\varepsilon+6t},
\]
which tends to $3$ only as $\varepsilon\to0$ \emph{and} $t\to0$. Conjecture \ref{conj:45} on
$[2,3]$ thus reduces, within this method, to the following.

\begin{question}\label{q:va}
Do the surfaces $X_p$ of \cite[Thm.~6.3]{RU} carry very ample line bundles $\sig_p$ with
\[
\sig_p^2=o\big(c_2(X_p)\big)\qquad\text{and}\qquad \sig_p\cdot K_{X_p}=o\big(c_2(X_p)\big)?
\]
\end{question}

A positive answer would give Conjecture \ref{conj:45} on all of $[1,3]$, by Theorem
\ref{thm:criterion} and Lemma \ref{lem:Lambda}. We do not know whether such a $\sig_p$
exists. The very ample classes we can write down on $X_p$ are pulled back from a blow-up of
$\PP^2$ along a morphism of degree $p$, and are too large in both respects; indeed the Hodge
index theorem gives $\sig_p\cdot K_{X_p}\ge\sqrt{\sig_p^2\,K_{X_p}^2}$, so the second
condition fails as soon as the first does.

\end{document}